\documentclass[preprint,10pt]{elsarticle}
\usepackage{amssymb,amsmath,graphicx,amscd,mathrsfs}
\usepackage{color,xcolor,amsmath}
\usepackage{amsthm}
\usepackage{graphicx}
\usepackage{mathrsfs}
\usepackage{float}
\usepackage{amsfonts,amssymb}
\usepackage{dsfont}
\usepackage{pifont}
\usepackage{hyperref}
\usepackage{multirow}
\usepackage{placeins}
\usepackage{subfigure}
\usepackage[title]{appendix}
\numberwithin{equation}{section}

\def\3bar{{|\hspace{-.02in}|\hspace{-.02in}|}}
\def\E{{\mathcal{E}}}
\def\T{{\mathcal{T}}}

\def\dQ{{\mathbb{Q}}}
\def\bQ{{\mathbf{Q}}}

\def\b0{\boldsymbol{0}}
\def\bphi{\boldsymbol{\phi}}            
\def\bpsi{\boldsymbol{\psi}}            
\def\sumT{\sum_{T\in\mathcal{T}_h}}     
\def\sumE{\sum_{e\in\E_h^0}}            

\def\bw{{\mathbf{w}}}
\def\bu{{\mathbf{u}}}
\def\bv{{\mathbf{v}}}
\def\bn{{\mathbf{n}}}
\def\be{{\mathbf{e}}}
\def\bf{{\mathbf{f}}}
\def\bg{{\mathbf{g}}}

\newtheorem{theorem}{Theorem}[section]
\newtheorem{algorithm}[theorem]{Conforming Discontinuous Galerkin Algorithm}
\newtheorem{lemma}{Lemma}[section]

 \newcommand{\eps}{\varepsilon}

 \newcommand{\Real}{\mathbb{R}}

 \newcommand{\set}[1]{\left\{#1\right\}}
 \newcommand{\seq}[1]{\left<#1\right>}

 \newcommand{\trb}[1]{|\!|\!|#1|\!|\!|}

\allowdisplaybreaks 

\usepackage{amssymb}

\journal{journal of computational and applied mathematics}

\begin{document}

\begin{frontmatter}

\title{A conforming discontinuous Galerkin finite element method for Brinkman equations}

\author[mymainaddress]{Haoning Dang}
\author[mymainaddress]{Qilong Zhai\corref{mycorrespondingauthor}}
\cortext[mycorrespondingauthor]{Corresponding author}
\ead{zhaiql@jlu.edu.cn}
\author[mythirdaddress]{Zhongshu Zhao}
\address[mymainaddress]{School of Mathematics, Jilin University,  Changchun, China.}
\address[mythirdaddress]{School of Mathematical Sciences and Institute of Natural Sciences, Shanghai Jiao Tong University, Shanghai, China.} 

\begin{abstract}
 
In this paper, we present a conforming discontinuous Galerkin (CDG) finite element method for Brinkman equations. The velocity stabilizer is removed by employing the higher degree polynomials to compute the weak gradient. The theoretical analysis shows that the CDG method is actually stable and accurate for the Brinkman equations. Optimal order error estimates are established in $H^1$ and $L^2$ norm. Finally, numerical experiments verify the stability and accuracy of the CDG numerical scheme.
\end{abstract}

%


\end{frontmatter}


\section{Introduction}
The Brinkman equations frequently appears in modeling the incompressible flow of a viscous fluid in complex porous medium. These equations extend Darcy's law to describe the dissipation of kinetic energy caused by viscous forces, similarly to the Navier-Stokes equations \cite{Brinkman1949ACO}. Brinkman equations are also applied in many other fields, such as environmental science, geophysics, petroleum engineering, biotechnology, and so on \cite{Avramenko2020,FERREIRA2020103696,fractured,CarbonateReservoirs,biotechnology}.  

Mathematically speaking, the Brinkman equations combine the Darcy equations and the Stokes equations by a highly varied parameter. For simplicity, we consider the following Brinkman equations in a bounded polygonal domain $\Omega\in\Real^2$: find the unknown fluid velocity $\bu$ and pressure $p$ satisfying 
\begin{align}
    -\mu\Delta\bu+\mu\kappa^{-1}\bu+\nabla p&=\bf\quad\text{in}~\Omega, \label{Eq:Brinkman1}\\ 
    \nabla\cdot\bu&=0\quad\text{in}~\Omega, \label{Eq:Brinkman2}\\
    \bu&=\b0\quad\text{on}~\partial\Omega, \label{Eq:Brinkman3}    
\end{align}
where $\kappa$ is permeability tensor, and $\mu$ is the fluid viscosity coefficient. For the convenience of analysis, we assume the permeability tensor $\kappa$ is piecewise constant. Since $\mu$ can be used to scale the solution $\bu$, we can take $\mu=1$ for simplicity. $\bf$ represents the momentum source term. In addition, assume that there exist two positive constants $\lambda_1$ and $\lambda_2$ such that
\begin{align}
    \lambda_{1}\xi^{t}\xi\le\xi^{t}\kappa^{-1}\xi\le\lambda_{2}\xi^{t}\xi,\quad\forall\xi\in\Real^{d}. \label{prop_kappa}
\end{align}

The main challenge of designing the numerical algorithm comes from the different regularity requirements of the velocity in such two extreme cases: the variational form of Brinkman equations in the Stokes limit requires $H^1$-regularity, while the Darcy limit requires $H(\text{div})$-regularity. To overcome this difficulty, many scholars have made extensive research. A natural attempt is to modify the existing Stokes or Darcy elements. We can easily find corresponding work related to Stokes based elements \cite{Stokesbased1,Stokesbased2,Stokesbased3} and Darcy based elements \cite{Darcybased2,Darcybased}. Another strategy is to design new formulations for Brinkman equations, such as the dual-mixed formulation \cite{dualmixed,doi:10.1142/S0218202511005726}, the pseudostress-velocity formulation \cite{QIAN2020113177} and the vorticity-velocity-pressure formulation \cite{velocity-vorticity-pressure}. In addition, some new numerical methods are introduced to Brinkman equations, such as weak Galerkin methods \cite{BRINK1,BRINK2}, virtual element methods \cite{doi:10.1142/S0218202517500142,Wang2021,ZHANG2021126325} and much more.

The purpose of this paper is to introduce a new conforming discontinuous Galerkin (CDG) method for Brinkman equations. The CDG method based on the weak Galerkin (WG) method proposed in \cite{ELLIPTIC}, is first proposed by Ye and Zhang in 2020 \cite{Ye2020110}. It retains the key idea of WG method, which uses the weak differential operators to approximate the classical differential operators in the variational form. In \cite{Ye2020110}, the authors prove that no stabilizer is required for Poisson problem when the local Raviart-Thomas (RT) element is used to approximate the classic gradient operator. However, we know that the RT elements are only applicable to triangular and rectangular meshes. In subsequent work \cite{JMIN1}, they found that the stabilizer term can be removed from the numerical scheme constructed on polygon meshes by raising the degree of polynomial approximating the discrete weak gradient operator. The CDG method has been applied to Stokes equations \cite{cdgstokes}, elliptic interface problem \cite{CDGellipticinterface} and linear elasticity interface problem \cite{CDGLinearElasticityInterface}.

In this paper, we consider the same variational form based on gradient-gradient operators as \cite{BRINK2}: find the unknown functions $\bu\in[H^1_0(\Omega)]^d$ and $p\in L^2_0(\Omega)$ satisfying
\begin{align}
    (\nabla\bu,\nabla\bv)+(\kappa^{-1}\bu,\bv)+(\nabla p,\bv)&= (\bf,\bv), &&\forall \bv\in[H^1_0(\Omega)]^d, \label{VF2.1}\\
    (\nabla q,\bu)&=0, &&\forall q\in L^2_0(\Omega). \label{VF2.2}
\end{align}
We construct the conforming discontinuous Galerkin scheme to discrete this variational form on polygon meshes. The stable term for the velocity $\bu$ is removed by a new definition of weak gradient operator, thus our numerical formulation is simpler compared with the standard WG method \cite{BRINK2}. Furthermore, we prove the well-posedness of the CDG scheme and derive the optimal error estimates for velocity and pressure, which implies the optimal convergence order for both the Stokes and Darcy dominated problems. Some numerical experiments are provided to verify our theoretical analysis.

The rest of the paper is organized as follows. In Section \ref{WEAKOPERATOR}, we define two discrete weak gradient operators and 
construct the conforming discontinuous Galerkin scheme for Brinkman equations. Then, the well-posedness is proved in Section \ref{CDG:SCHEME}. The error equations for the CDG scheme are established in Section \ref{CDG:ERROR_EQUATION}. And we prove optimal error estimates for both velocity and pressure in $H^1$ and $L^2$ norms in Section \ref{CDG:ERROR_ESTIMATE}. In Section \ref{EXAMPLE}, we present some numerical experiments to verify the stability and accuracy of the CDG scheme. 

\section{Discrete Weak Gradient Operators} \label{WEAKOPERATOR}
In this section, we define two discrete weak gradient operators that we're going to use later.

Let $\T_h$ be a polygonal or polyhedral partition of the domain $\Omega$ and $\E_h$ be the set of all edges or faces in $\T_h$. Assume that all cells in $\T_h$ are closed and simply connected, and satisfy some specific shape regular conditions in \cite{MESH2}. Denote the set of all interior edges or faces by $\E_h^{0}=\E_h\backslash\partial\Omega$. For each $T\in\T_h$, $e\subset\partial T$, let $h_T$ and $h_e$ be the diameter of $T$ and $e$, respectively. And we define the size of $\T_h$ as $h=\max\limits_{T\in\T_h}h_T$. 

For a given integer $k\ge1$, the space of polynomial with degree no more than $k$ on a cell $T$ denotes by $P_k(T)$. We define the space for the vector-valued functions as 
\begin{align*}
    V_h=\set{\bv\in [L^2(\Omega)]^d:\bv|_T \in[P_k(T)]^d, \forall T\in\T_h}.
\end{align*}
Denote by $V_h^0$ the subspace of $V_h$ that
\begin{align*}
    V_h^0=\set{\bv:\bv\in V_h, \bv=\b0~\text{on}~\partial\Omega}.
\end{align*}

For the scalar-valued functions, we define
\begin{align*}
    W_h=\set{q: q\in L^2_0(\Omega),q|_T\in P_{k-1}(T)}.
\end{align*}

Let $T_1$ and $T_2$ be two cells in $\T_h$ sharing $e\in\E_h^0$, $\bn_1$ and $\bn_2$ be the unit outward normal vectors of $T_1$ and $T_2$ on $e$. In particular, when $e\subset\partial\Omega$, we denote the unit outward normal vector of $T$ on $e$ by $\bn_e$. For a vector-valued function $\bv\in V_h+[H^1(\Omega)]^d$, we define the average $\set{\cdot}$ and jump $[\cdot]$ as follows
\[\set{\bv}=\left\{
    \begin{array}{ll}
        \frac{1}{2}(\bv|_{T_1}+\bv|_{T_2}) & ~e\in\E_h^0,\\
        \bv|_e & ~e\subset\partial\Omega,
    \end{array}\right.\quad
[\bv]=\left\{
    \begin{array}{ll}
        \bv|_{T_1}\cdot\bn_1+\bv|_{T_2}\cdot\bn_2& ~e\in\E_h^0, \\ 
        \bv|_e\cdot\bn_e & ~e\subset\partial\Omega.
    \end{array}\right.\]
For a scalar-valued function $q\in W_h$, the average $\set{\cdot}$ and the jump $[\![\cdot]\!]$ are
\[\set{q}=\left\{
    \begin{array}{ll}
        \frac{1}{2}(q|_{T_1}+q|_{T_2}) & ~e\in\E_h^0,\\
        q|_e & ~e\subset\partial\Omega,
    \end{array}\right. \quad
[\![q]\!]=\left\{
    \begin{array}{ll}
        q|_{T_1}\bn_1+q|_{T_2}\bn_2& ~e\in\E_h^0, \\ 
        q|_e\bn_e & ~e\subset\partial\Omega.
    \end{array}\right.\]

In addition, for $\bv\in V_h^0+[H^1_0(\Omega)]^d$, the following equations hold true:
\begin{align} \label{v_averjump}
    \|\bv-\set{\bv}\|_e=\|[\bv]\|_e,\quad\text{if}~e\subset\partial\Omega, \quad\|\bv-\set{\bv}\|_e=\frac{1}{2}\|[\bv]\|_e,\quad\text{if}~e\in\mathcal{E}_h^0.
\end{align}
Similarly, for $q\in W_h$, we have
\begin{align} \label{q_averjump}
    \|q-\set{q}\|_e=0,\quad\text{if}~e\subset\partial\Omega, \quad\|q-\set{q}\|_e=\frac{1}{2}\|[\![q]\!]\|_e,\quad\text{if}~e\in\mathcal{E}_h^0.
\end{align}

Then we give the definition of the discrete weak gradient operators. For a vector-valued function $\bv\in V_h+[H^1(\Omega)]^d$, the discrete weak gradient $\nabla_w\bv$ on each cell $T$ is a unique polynomial function in $[P_j(T)]^{d\times d}(j>k)$ satisfying 
\begin{align}
    (\nabla_w\bv,\tau)_T=-(\bv,\nabla\cdot\tau)_T+\seq{\set{\bv},\tau\bn}_{\partial T}, \quad \forall\tau\in[P_j(T)]^{d\times d}. \label{CDG:grad_v}
\end{align}
Note that $j$ depends on $k$ and $n$ is the number of edges of polygon cell $T$. For a polygonal mesh, $j=n+k-1$ \cite{JMIN1}, and in particular, $j=k+1$ when the domain is partitioned into triangles \cite{JMIN2}. 

Similarly, for a scalar-valued function $q\in W_h$, the discrete weak gradient $\tilde{\nabla}_w q$ on each cell $T$ is a unique polynomial function in $[P_k(T)]^d$ satisfying 
\begin{align}
    (\tilde{\nabla}_w q,\bphi)_T=-(q,\nabla\cdot\bphi)_T+\seq{\set{q},\bphi\cdot\bn}_{\partial T}, \quad\forall\bphi\in[P_k(T)]^d. \label{CDG:grad_q}
\end{align}

For simplicity of notations, we introduce three bilinear forms as follows: 
\begin{align}
    a(\bv,\bw)&=(\nabla_w\bv,\nabla_w\bw)+(\kappa^{-1}\bv,\bw), \label{CDG:bilineara}\\
    b(\bv,q)&=(\bv,\tilde{\nabla}_w q), \label{CDG:bilinearb}\\
    s(p,q)&=\sumE h\seq{[\![p]\!],[\![q]\!]}_e. \label{CDG:bilinears}
 \end{align}
 
Now we have the following conforming discontinuous Galerkin finite element scheme for Brinkman equations (\ref{Eq:Brinkman1})-(\ref{Eq:Brinkman3}). 
\begin{algorithm}
    Find $\bu_h\in V_h^0$ and $p_h\in W_h$ such that
    \begin{align}
        a(\bu_h,\bv)+b(\bv,p_h)&=(\bf,\bv), &&\forall\bv\in V_h^0,\label{CDG:algo1}\\
        b(\bu_h,q)-s(p_h,q)&=0, &&\forall q\in W_h. \label{CDG:algo2}
    \end{align}
\end{algorithm}
\section{The Well-Posedness of CDG Scheme} \label{CDG:SCHEME}
In this section, we show that the CDG numerical scheme (\ref{CDG:algo1})-(\ref{CDG:algo2}) has a unique solution.

In order to analyse the well-posedness of the CDG scheme, we first define the following tri-bar norm. For a vector-valued function $\bv\in V_h+[H^1(\Omega)]^d$, 
\begin{align}
    \trb{\bv}^2=a(\bv,\bv)=\|\nabla_w\bv\|^2+\|\kappa^{-\frac{1}{2}}\bv\|^2. \label{CDG:trb}
\end{align}
It is obvious that $\trb{\cdot}$ indeed provides a norm in $V_h$.

For a scalar-valued function $q\in W_h$, we use the following norms in the rest of this paper,
\begin{align}
    \trb{q}^2_1&=\|\kappa^{\frac{1}{2}}\tilde{\nabla}_wq\|^2+\sumE h^{-1}\|[\![q]\!]\|_e^2,\label{CDG:q_1}\\
    \| q\|_h^2&=\sumE h\|[\![q]\!]\|_e^2. \label{CDG:q_h}
\end{align}

By the definition of $\trb{\cdot}$ and the Cauchy-Schwarz inequality, the following the boundedness for the bilinear form $a(\cdot,\cdot)$ holds.
\begin{lemma} \label{CDG:lemma_bound}
    For any $\bv,\bw\in V_h+[H^1(\Omega)]^d$, we have 
    \begin{align*} 
        |a(\bv,\bw)|\le\trb{\bv}\trb{\bw}.
    \end{align*}
\end{lemma}

Next, we present the inf-sup condition of $b(\cdot,\cdot)$.
\begin{lemma} \label{CDG:lemma_infsup}
    For any $q\in W_h$, there exist constants $C_1$ and $C_2$ independent of $h$ such that
    \begin{align} \label{CDG:infsup}
        \underset{\bv\in V_h}{sup}\frac{|b(\bv,q)|}{\trb{\bv}}\ge C_1h\trb{q}_1-C_2\|q\|_h.
    \end{align}
\end{lemma}
\begin{proof}
    For any $\bv\in V_h+[H^1(\Omega)]^d$, from the definition of $\nabla_w$, we get
    \begin{align*}
    \|\nabla_w\bv\|^2=\sumT((\nabla\bv,\nabla\bv)_T-\seq{\bv-\set{\bv},(\nabla_w\bv+\nabla\bv)\cdot\bn}_{\partial T}),
    \end{align*}
    Using the definition of $\set{\cdot}$, the following derivation holds true
    \begin{align*}
        &\sumT\|\bv-\set{\bv}\|_{\partial T}^2\\
        &\le\sum_{e\in\E_h}\left(\left\|\bv|_{T_1}-\frac{\bv|_{T_1}+\bv|_{T_2}}{2}\right\|_{\partial T_1\cap e}^2+\left\|\bv|_{T_2}-\frac{\bv|_{T_1}+\bv|_{T_2}}{2}\right\|_{\partial T_2\cap e}^2\right)\\
        &\le C\sum_{e\in\E_h}\left\|\frac{\bv|_{T_1}-\bv|_{T_2}}{2}\right\|_e^2\\
        &\le C\sum_{e\in\E_h}\left(\|\bv\|_{\partial T_1\cap e}^2+\|\bv\|_{\partial T_2\cap e}^2\right)\\
        &\le C\sumT\|\bv\|_{\partial T}^2. 
    \end{align*}
    Using the above inequality, the trace inequality (\ref{trace_ineq}) and the inverse inequality (\ref{inverse_ineq}), we obtain
    \begin{align*}
        \|\nabla_w\bv\|^2&\le\sumT(\|\nabla\bv\|_T^2+\|\bv-\set{\bv}\|_{\partial T}\|\nabla_w\bv+\nabla\bv\|_{\partial T})\\
        &\le\sumT(\|\nabla\bv\|_T^2+h^{-\frac{1}{2}}\|\bv-\set{\bv}\|_{\partial T}(\|\nabla_w\bv\|_T+\|\nabla\bv\|_T))\\
        &\le\sumT(\|\nabla\bv\|_T^2+Ch^{-\frac{1}{2}}\|\bv\|_{\partial T}(\|\nabla_w\bv\|_T+\|\nabla\bv\|_T))\\
        &\le\sumT(\|\nabla\bv\|_T^2+\frac{1}{2}\|\nabla_w\bv\|_T^2+\frac{1}{2}\|\nabla\bv\|_T^2+Ch^{-1}\|\bv\|_{\partial T}^2)\\
        &\le\sumT\left(\frac{1}{2}\|\nabla_w\bv\|_T^2+Ch^{-2}\|\bv\|_T^2\right).
    \end{align*}
    Taking $\bv=\kappa\tilde{\nabla}_w q\in V_h$ and using (\ref{prop_kappa}), we get the following estimate
    \begin{align*}
        \trb{\bv}^2&=\|\nabla_w\bv\|^2+\|\kappa^{-\frac{1}{2}}\bv\|^2\\
        &\le Ch^{-2}\|\bv\|^2+\|\kappa^{-\frac{1}{2}}\bv\|^2\\
        &\le Ch^{-2}\|\kappa\tilde{\nabla}_w q\|^2+\|\kappa^{\frac{1}{2}}\tilde{\nabla}_w q\|^2\\
        &\le Ch^{-2}\lambda_1^{-1}\|\kappa^{\frac{1}{2}}\tilde{\nabla}_w q\|^2+\|\kappa^{\frac{1}{2}}\tilde{\nabla}_w q\|^2\\
        &\le Ch^{-2}\trb{q}_1^2.
    \end{align*}
    Therefore, we have
    \begin{align*}
        \frac{b(\bv,q)}{\trb{\bv}}&=\frac{\trb{q}_1^2-h^{-2}\|q\|_h^2}{Ch^{-1}\trb{q}_1}\\
        &\ge\frac{\trb{q}_1^2-h^{-1}\trb{q}_1\|q\|_h}{Ch^{-1}\trb{q}_1}\\
        &\ge C_1h\trb{q}_1-C_2\|q\|_h.
    \end{align*}
    
\end{proof} 

\begin{lemma}
    The conforming discontinuous Galerkin finite element scheme (\ref{CDG:algo1})-(\ref{CDG:algo2}) has a unique solution.
\end{lemma}
\begin{proof}
    Consider the corresponding homogeneous equation $\bf=\b0$, let $\bv=\bu_h$ in (\ref{CDG:algo1}) and $q=p_h$ in (\ref{CDG:algo2}). Then subtracting (\ref{CDG:algo2}) form (\ref{CDG:algo1}), we have
    \begin{align*}
        \trb{\bu_h}^2+\|p_h\|_h^2=a(\bu_h,\bu_h)+s(p_h,p_h)=0,
    \end{align*}
    which implies $\bu_h=\b0$ and $\|p_h\|_h=0$.
    
    Let $\bu_h=\b0$ and $\bf=\b0$ in (\ref{CDG:algo1}), we have $b(\bv,p_h)=0$ for any $\bv\in V_h^0$. According to the definition of $b(\cdot,\cdot)$ and $\|p_h\|_h=0$, let $\bv=\nabla p_h$ on each cell $T$, it holds that 
    \begin{align*}
        0=b(\bv,p_h)&=\sumT(\bv,\tilde{\nabla}_wp_h)_T\\
        &=\sumT(-(p_h,\nabla\cdot\bv)_T+\seq{\{p_h\},\bv\cdot\bn}_{\partial T})\\
        &=\sumT(-(p_h,\nabla\cdot\bv)_T+\seq{p_h,\bv\cdot\bn}_{\partial T}-\seq{p_h-\{p_h\},\bv\cdot\bn}_{\partial T})\\
        &=\sumT(\bv,\nabla p_h)_T-\sumE\seq{[\![p_h]\!],\{\bv\}}_{\partial T}\\
        &=\sumT(\bv,\nabla p_h)_T=\sumT(\nabla p_h,\nabla p_h)_T.
    \end{align*}
    Then, we have $\nabla p_h=0$ on each cell $T$. Since $p_h$ is continuous and $p_h\in L^2_0(\Omega)$, we have $p_h=0$ and complete the proof of the lemma.
    
\end{proof}

\section{Error Equations} \label{CDG:ERROR_EQUATION}
In this section, we establish the error equations between the numerical solution and the exact solution.

Let $Q_h$, $\bQ_h$ and $\dQ_h$ be the standard $L^2$ projection operators onto $[P_k(T)]^d$, $[P_j(T)]^{d\times d}$ and $P_{k-1}(T)$, respectively. 

First, we recall some properties about the projection operators.
\begin{lemma}
    For the projection operators $Q_h$, $\bQ_h$ and $\dQ_h$, the following properties hold
    \begin{align}
         \nabla_w\bv &=\bQ_h(\nabla\bv), \quad\forall\bv\in [H^1(\Omega)]^d, \label{CDG:Eq_proj1}\\
         (\tilde{\nabla}_w(\dQ_hq),\bphi)_T &=(Q_h(\nabla q),\bphi)_T-\seq{q-\dQ_hq,\bphi\cdot\bn}_{\partial T},\label{CDG:Eq_proj2}\\
        &\qquad\qquad\forall q\in H^1(\Omega), ~\forall\bphi\in[P_k(T)]^d. \nonumber
    \end{align}
\end{lemma}
\begin{proof}
    For any $\tau\in[P_j(T)]^{d\times d}$, we have
   \begin{align*}
        (\nabla_w\bv,\tau)_T&=-(\bv,\nabla\cdot\tau)_T+\seq{\set{\bv},\tau\cdot\bn}_{\partial T}\\
        &=-(\bv,\nabla\cdot\tau)_T+\seq{\bv,\tau\cdot\bn}_{\partial T}\\
        &=(\nabla\bv,\tau)_T\\
        &=(\bQ_h(\nabla\bv),\tau)_T.
    \end{align*}
    Similarly, for any $\bphi\in[P_k(T)]^d$, we have
    \begin{align*}
        (\tilde{\nabla}_w(\dQ_hq),\bphi)_T&=-(\dQ_hq,\nabla\cdot\bphi)_T+\seq{\set{\dQ_hq},\bphi\cdot\bn}_{\partial T}\\
        &=-(q,\nabla\cdot\bphi)_T+\seq{\dQ_hq,\bphi\cdot\bn}_{\partial T}\\
        &=-(q,\nabla\cdot\bphi)_T+\seq{q,\bphi\cdot\bn}_{\partial T}-\seq{q-\dQ_hq,\bphi\cdot\bn}_{\partial T}\\
        &=(\nabla q,\bphi)_T-\seq{q-\dQ_hq,\bphi\cdot\bn}_{\partial T}\\
        &=(Q_h(\nabla q),\bphi)_T-\seq{q-\dQ_hq,\bphi\cdot\bn}_{\partial T},
    \end{align*}
    which completes the proof of the lemma.
\end{proof}

Let $\be_h=Q_h\bu-\bu_h$ and $\eps_h=\dQ_hp-p_h$ be the error functions, where $(\bu;p)$ be the solution of (\ref{Eq:Brinkman1})-(\ref{Eq:Brinkman3}) and $(\bu_h;p_h)$ be the solution of (\ref{CDG:algo1})-(\ref{CDG:algo2}). We shall derive the error equations that $\be_h$ and $\eps_h$ satisfy.

\begin{lemma} \label{CDG:lemma_error}
    For any $\bv\in V_h^0$ and $q\in W_h$, the following equations hold true
    \begin{align}
        a(\be_h,\bv)+b(\bv,\eps_h)&=-l_1(\bu,\bv)+l_2(\bu,\bv)-l_3(p,\bv),\label{CDG:Eq_err1}\\
        b(\be_h,q)-s(\eps_h,q)&=l_4(\bu,q)-s(\dQ_hp,q),  \label{CDG:Eq_err2}
    \end{align}
    where 
    \begin{align*}
        l_1(\bu,\bv)&=\sumT(\nabla_w(\bu-Q_h\bu),\nabla_w\bv)_T,\\
        l_2(\bu,\bv)&=\sumT\seq{(\nabla\bu-\bQ_h\nabla\bu)\cdot\bn,\bv-\set{\bv}}_{\partial T},\\
        l_3(p,\bv)&=\sumT\seq{p-\dQ_hp,\bv\cdot\bn}_{\partial T},\\
        l_4(\bu,q)&=\sumT\seq{(\bu-Q_h\bu)\cdot\bn,q-\set{q}}_{\partial T},\\
        s(\dQ_hp,q)&=\sumE h\seq{[\![\dQ_hp]\!],[\![q]\!]}.
    \end{align*}
\end{lemma}
\begin{proof}
    Testing (\ref{Eq:Brinkman1}) by $\bv\in V_h^0$ yields 
    \begin{align*}
        -(\Delta\bu,\bv)+(\kappa^{-1}\bu,\bv)+(\nabla p,\bv)=(\bf,\bv).
    \end{align*}
    Applying the definition of projection operators $Q_h$, $\bQ_h$ and $\dQ_h$, the definition of the weak gradient $\nabla_w$ and $\tilde{\nabla}_w$ and (\ref{CDG:Eq_proj1}), we get
    \begin{align*}
        -(\Delta\bu,\bv)&=\sumT((\nabla\bu,\nabla\bv)_T-\seq{\nabla\bu\cdot\bn,\bv}_{\partial T})\\
        &=\sumT((\bQ_h\nabla\bu,\nabla\bv)_T-\seq{\nabla\bu\cdot\bn,\bv}_{\partial T})\\
        &=\sumT(-(\bv,\nabla\cdot(\bQ_h\nabla\bu))_T+\seq{\bv,(\bQ_h\nabla\bu-\nabla\bu)\cdot\bn}_{\partial T})\\
        &=\sumT((\nabla_w\bv,\bQ_h\nabla\bu)_T+\seq{\bv-\set{\bv},(\bQ_h\nabla\bu-\nabla\bu)\cdot\bn}_{\partial T})\\
        &=\sumT((\nabla_w\bu,\nabla_w\bv)_T+\seq{\bv-\set{\bv},(\bQ_h\nabla\bu-\nabla\bu)\cdot\bn}_{\partial T})\\
        &=\sumT(\nabla_w(Q_h\bu),\nabla_w\bv)_T+l_1(\bu,\bv)-l_2(\bu,\bv). \\
    \end{align*}
    According to the definition of $Q_h$ and (\ref{CDG:Eq_proj2}), we have
    \begin{align*}
        (\nabla p,\bv)&=\sumT(\nabla p,\bv)_T\\
        &=\sumT(Q_h(\nabla p),\bv)_T\\
        &=\sumT(\tilde{\nabla}_w(\dQ_hp),\bv)_T+\sumT\seq{p-\dQ_hp,\bv\cdot\bn}_{\partial T}\\
        &=\sumT(\tilde{\nabla}_w(\dQ_hp),\bv)_T+l_3(p,\bv),
    \end{align*}
    and
    \begin{align*}
        (\kappa^{-1}\bu,\bv)&=(\bu,\kappa^{-1}\bv)=(Q_h\bu,\kappa^{-1}\bv)=(\kappa^{-1}Q_h\bu,\bv).
    \end{align*}
    Using the definition of $a(\cdot,\cdot)$ and $b(\cdot,\cdot)$ and the above equations, we have
    \begin{align}
        a(Q_h\bu,\bv)+b(\bv,\dQ_hp)+l_1(\bu,\bv)-l_2(\bu,\bv)+l_3(p,\bv)=(\bf,\bv). \label{CDG:Eq_test1}
    \end{align}
    Subtracting (\ref{CDG:algo1}) from (\ref{CDG:Eq_test1}), we arrive at
    \begin{align*}
        a(\be_h,\bv)+b(\bv,\eps_h)+l_1(\bu,\bv)-l_2(\bu,\bv)+l_3(p,\bv)=0, 
    \end{align*}
    which completes the proof of (\ref{CDG:Eq_err1}).
 
    Similarly, testing (\ref{Eq:Brinkman2}) by $q\in W_h$ yields
    \begin{align*}
        0=(\nabla\cdot\bu,q)&=\sumT(-(\nabla q,\bu)_T+\seq{ q,\bu\cdot\bn}_{\partial T})\\
        &=\sumT(-(\nabla q,Q_h\bu)_T+\seq{ q,\bu\cdot\bn}_{\partial T})\\
        &=\sumT((q,\nabla\cdot(Q_h\bu))_T+\seq{ q,(\bu-Q_h\bu)\cdot\bn}_{\partial T})\\
        &=\sumT(-(Q_h\bu,\tilde{\nabla}_wq)_T,\seq{ q-\set{q},(\bu-Q_h\bu)\cdot\bn}_{\partial T}). 
    \end{align*}
    Using (\ref{CDG:algo2}), we have
    \begin{align*}
        \sumT(Q_h\bu-\bu_h,\tilde{\nabla}_wq)_T-\sumE h\seq{[\![p_h]\!],[\![q]\!]}_e=l_4(\bu,q).
    \end{align*}
    Adding $-s(\dQ_hp,q)$ to both sides of this equation, it follows that
    \begin{align*}
        \sumT(Q_h\bu-\bu_h,\tilde{\nabla}_wq)_T-\sumE h\seq{[\![\dQ_hp-p_h]\!],[\![q]\!]}_e=l_4(\bu,q)-s(\dQ_hp,q),
    \end{align*}
    which implies (\ref{CDG:Eq_err2}) and we complete the proof of this lemma.
\end{proof}
\section{Error Estimates} \label{CDG:ERROR_ESTIMATE}
In this section, we aim to deriving the error estimates to the numerical scheme (\ref{CDG:algo1})-(\ref{CDG:algo2}). 

\begin{theorem} \label{CDG:Thm_H1}
    Let $(\bu;p)\in[H^1_0(\Omega)\cap H^{k+1}(\Omega)]^d\times(L^2_0(\Omega)\cap H^k(\Omega))$ be the exact solution of (\ref{Eq:Brinkman1})-(\ref{Eq:Brinkman3}) and $(\bu_h;p_h)\in V_h^0\times W_h$ be the solution of (\ref{CDG:algo1})-(\ref{CDG:algo2}), respectively. Then there exists a constant $C$ such that
    \begin{align}
        \trb{\be_h}+\|\eps_h\|_h&\le Ch^k(\|\bu\|_{k+1}+\|p\|_k), \label{CDG:Ineq_H11}\\
        \trb{\eps_h}_1&\le Ch^{k-1}(\|\bu\|_{k+1}+\|p\|_k). \label{CDG:Ineq_H12}
    \end{align}
\end{theorem}
\begin{proof}
    Let $\bv=\be_h$ in (\ref{CDG:Eq_err1}) and $q=\eps_h$ in (\ref{CDG:Eq_err2}), we have
    \begin{align*}
        \trb{\be_h}^2+\|\eps_h\|_h^2=-l_1(\bu,\be_h)+l_2(\bu,\be_h)-l_3(p,\be_h)-l_4(\bu,\eps_h)+s(\dQ_hp,\eps_h).
    \end{align*}
    From the estimates (\ref{CDG:Ineq_l4})-(\ref{CDG:Ineq_s}), we obtain
    \begin{align*}
        \trb{\be_h}^2+\|\eps_h\|_h^2\le Ch^k(\|\bu\|_{k+1}+\|p\|_k)(\trb{\be_h}+\|\eps_h\|_h),
    \end{align*}
    which implies (\ref{CDG:Ineq_H11}).
    
    Let $p=\eps_h$ in (\ref{CDG:Eq_err1}), from the inf-sup condition (\ref{CDG:infsup}), Lemma (\ref{CDG:lemma_bound}) and (\ref{CDG:Ineq_l4})-(\ref{CDG:Ineq_l6}), it follows that
    \begin{align*}
        (C_1h\trb{\eps_h}_1-C_2\|\eps_h\|_h)\trb{\bv}&\le|b(\bv,\eps_h)|\\
        &\le a(\be_h,\bv)+l_1(\bu,\bv)+l_2(\bu,\bv)+l_3(p,\bv)\\
        &\le\trb{\be_h}\trb{\bv}+Ch^k(\|\bu\|_{k+1}+\|p\|_k)\trb{\bv}.
    \end{align*}
    Combining the estimates above, we arrive at
    \begin{align*}
        h\trb{\eps_h}_1&\le\trb{\be_h}+\|\eps_h\|_h+Ch^k(\|\bu\|_{k+1}+\|p\|_k)\\
        &\le Ch^k(\|\bu\|_{k+1}+\|p\|_k),
    \end{align*}
    which completes the proof.
\end{proof}

In order to obtain $L^2$ error estimate, we consider the following dual problem: seek $(\bpsi;\xi)\in[H^2(\Omega)]^d\times H^1(\Omega)$ satisfying 
\begin{align}
    -\Delta\bpsi+\kappa^{-1}\bpsi+\nabla \xi &= \be_h \quad \text{in}~\Omega, \label{CDG:Eq_dual1}\\ 
    \nabla \cdot\bpsi &= 0 \quad \text{in}~\Omega, \label{CDG:Eq_dual2}\\
    \bpsi &= \b0 \quad \text{on}~\partial\Omega. \label{CDG:Eq_dual3}
\end{align}
Assume that the following regularity condition holds,
\begin{align}
    \|\bpsi\|_2+\|\xi\|_1\le C\|\be_h\|. \label{CDG:pri_cond}
\end{align}

\begin{theorem} \label{CDG:Thm_L2}
    Let $(\bu;p)\in[H^1_0(\Omega)\cap H^{k+1}(\Omega)]^d\times(L^2_0(\Omega)\cap H^k(\Omega))$ be the exact solution of (\ref{Eq:Brinkman1})-(\ref{Eq:Brinkman3}) and $(\bu_h;p_h)\in V_h^0\times W_h$ be the solution of (\ref{CDG:algo1})-(\ref{CDG:algo2}), respectively. Then there exists a constant $C$ such that 
    \begin{align}
        \|\be_h\|\le Ch^{k+1}(\|\bu\|_{k+1}+\| p\|_k). \label{CDG:Ineq_L2}
    \end{align}
\end{theorem}
\begin{proof}
    Testing (\ref{CDG:Eq_dual1}) by $\be_h$ yields
    \begin{align*}
        \|\be_h\|^2=(\be_h,\be_h)=-(\Delta\bpsi,\be_h)+(\kappa^{-1}\bpsi,\be_h)+(\nabla\xi,\be_h).
    \end{align*}
    Similar to the proof of Lemma \ref{CDG:lemma_error}, we have
    \begin{align*}
        -(\Delta\bpsi,\be_h)&=(\nabla_w(Q_h\bpsi),\nabla_w\be_h)+l_1(\bpsi,\be_h)-l_2(\bpsi,\be_h),\\
    \end{align*}
    and
    \begin{align*}
        (\nabla\xi,\be_h)&=(\tilde{\nabla}_w(\dQ_h\xi),\be_h)+l_3(\xi,\be_h).\\
    \end{align*}
    Combining the definition of $a(\cdot,\cdot)$ and $b(\cdot,\cdot)$ gives
    \begin{align}
        \|\be_h\|^2=a(Q_h\bpsi,\be_h)+b(\be_h,\dQ_h\bpsi)+l_1(\bpsi,\be_h)-l_2(\bpsi,\be_h)+l_3(\xi,\be_h). \label{CDG:Eq_duala}
    \end{align}
    Similarly, testing (\ref{CDG:Eq_dual2}) by $\eps_h$ yields
    \begin{align}
        b(Q_h\bpsi,\eps_h)=l_4(\bpsi,\eps_h). \label{CDG:Eq_dualb}
    \end{align}
     Let $\bv=Q_h\bpsi$ and $q=\dQ_h\xi$ in (\ref{CDG:Eq_err1}) and (\ref{CDG:Eq_err2}), we have
    \begin{align}
        a(\be_h,Q_h\bpsi)+b(Q_h\bpsi,\eps_h)&=-l_1(\bu,Q_h\bpsi)+l_2(\bu,Q_h\bpsi)-l_3(p,Q_h\bpsi),\label{CDG:Eq_dualc}\\
        b(\be_h,\dQ_h\xi)-s(\eps_h,\dQ_h\xi)&=l_4(\bu,\dQ_h\xi)-s(\dQ_hp,\dQ_h\xi). \label{CDG:Eq_duald}
    \end{align}
    With (\ref{CDG:Eq_duala})-(\ref{CDG:Eq_duald}), we obtain
    \begin{align*}
        \|\be_h\|^2=&-l_1(\bpsi,\be_h)+l_2(\bpsi,\be_h)-l_3(\xi,\be_h)+l_4(\bpsi,\eps_h)-s(\dQ_h\bpsi,\eps_h)\\
        &+(l_1(\bu,Q_h\bpsi)-l_2(\bu,Q_h\bpsi)+l_3(p,Q_h\bpsi)-l_4(\bu,\dQ_h\xi)+s(\dQ_hp,\dQ_h\xi)).
    \end{align*}
    
    Let $\hat{\bQ}_h$ be the projection operator from $[L^2(T)]^{d\times d}$ onto $[P_1(T)]^{d\times d}$. For any $q\in P_1(T)$, we have
    \begin{align*}
        (\hat{\bQ}_h\nabla\bpsi,q)_T=(\nabla\bpsi,q)_T=-(\bpsi,\nabla\cdot q)_T+\seq{\bpsi,q\bn}_{\partial T}=(\nabla_w\bpsi,q)_T=(\hat{\bQ}_h\nabla_w\bpsi,q)_T,
    \end{align*}
    which implies $\hat{\bQ}_h\nabla\bpsi$ is equal to $\hat{\bQ}_h\nabla_w\bpsi$ on each cell $T$. 
    
    According to the definition of $\nabla_w$, the following equation holds true
    \begin{align*}
        (\nabla_w(\bu-Q_h\bu),\hat{\bQ}_h\nabla_w\bpsi)_T=&-(\bu-Q_h\bu,\nabla\cdot\hat{\bQ}_h\nabla_w\bpsi)_T\\
        &+\langle\set{\bu-Q_h\bu},\hat{\bQ}_h\nabla_w\bpsi\cdot\bn\rangle_{\partial T}.
    \end{align*}
    
    Notice that $k$ is an integer not less than one, from the definition of the projection operator $Q_h$, we obtain
    \begin{align}
        (\nabla_w(\bu-Q_h\bu),\hat{\bQ}_h\nabla\bpsi)_T=(\nabla_w(\bu-Q_h\bu),\hat{\bQ}_h\nabla_w\bpsi)_T=0. \label{CDG:Eq_projQ1}
    \end{align}

    Then,  using the projection inequalities (\ref{proj_ineq1})-(\ref{proj_ineq2}), and summing over all the cells, we arrive at
    \begin{align*}
        &\sumT(\nabla_w(\bu-Q_h\bu),\nabla_w\bpsi)_T\\
        =&\sumT(\nabla_w(\bu-Q_h\bu),\nabla_w\bpsi-\hat{\bQ}_h\nabla\bpsi)_T\\
        =&\sumT(\nabla_w(\bu-Q_h\bu),\nabla\bpsi-\hat{\bQ}_h\nabla \bpsi)_T\\
        \le& \left(\sumT\|\nabla_w(\bu-Q_h\bu)\|^2_T\right)^\frac{1}{2}\left(\sumT\|\nabla\bpsi-\hat{\bQ}_h\nabla \bpsi\|^2_T\right)^\frac{1}{2}\\
        \le& Ch^{k+1}\|\bu\|_{k+1}\|\bpsi\|_2. 
    \end{align*}
    For $l_1(\bu,Q_h\bpsi)$, we get
    \begin{align*}
        &|l_1(\bu,Q_h\bpsi)|\\
        =&\Bigg|\sumT(\nabla_w(\bu-Q_h\bu),\nabla_w Q_h\bpsi)_T\Bigg|\\
        \le&\Bigg|\sumT(\nabla_w(\bu-Q_h\bu),\nabla_w\bpsi)_T\Bigg|+\Bigg|\sumT(\nabla_w(\bu-Q_h\bu),\nabla_w (Q_h\bpsi-\bpsi))_T\Bigg|\\
        \le& Ch^{k+1}\|\bu\|_{k+1}\|\bpsi\|_2.
    \end{align*}
    Similarly, we have
    \begin{align*}
        |l_2(\bu,Q_h\bpsi)|&=\Bigg|\sumT\seq{(\nabla\bu-\bQ_h\nabla\bu)\cdot\bn,Q_h\bpsi-\set{Q_h\bpsi}}_{\partial T}\Bigg|\\
        &=\Bigg|\sumT\seq{(\nabla\bu-\bQ_h\nabla\bu)\cdot\bn,Q_h\bpsi-\bpsi+\set{\bpsi-Q_h\bpsi}}_{\partial T}\Bigg|\\
        &\le C\left(\sumT h\|\nabla\bu-\bQ_h\nabla\bu\|_{\partial T}^2\right)^{\frac{1}{2}}\left(\sumT h^{-1}\|Q_h\bpsi-\bpsi\|_{\partial T}^2\right)^{\frac{1}{2}}\\
        &\le Ch^{k+1}\|\bu\|_{k+1}\|\bpsi\|_2.
    \end{align*}
    According the projection inequality (\ref{proj_ineq3}), we have
    \begin{align*}
        |l_3(p,Q_h\bpsi)|&=\Bigg|\sumT\seq{p-\dQ_hp,Q_h\bpsi\cdot\bn}_{\partial T}\Bigg|\\
        &=\Bigg|\sumT\seq{p-\dQ_hp,(Q_h\bpsi-\bpsi)\cdot\bn}_{\partial T}\Bigg|\\
        &\le C\left(\sumT h\|p-\dQ_hp\|_{\partial T}^2\right)^{\frac{1}{2}}\left(\sumT h^{-1}\|Q_h\bpsi-\bpsi\|_e^2\right)^{\frac{1}{2}}\\
        &\le Ch^{k+1}\|p\|_k\|\bpsi\|_2.
    \end{align*}
    In the above derivation, we have used the following fact
    \begin{align*}
        \sumT\seq{p-\dQ_hp,\bpsi\cdot\bn}_{\partial T}=0.
    \end{align*}
    Similar to $l_2(\bu,Q_h\bpsi)$, we get
    \begin{align*}
        |l_4(\bu,\dQ_h\xi)|\le Ch^{k+1}\|\bu\|_{k+1}\|\xi\|_1.
    \end{align*}
    Using the definition of $s(\cdot,\cdot)$ and the projection inequality (\ref{proj_ineq3}), we have
    \begin{align*}
        |s(\dQ_hp,\dQ_h\xi)|&=\Bigg|\sumE h\seq{[\![\dQ_hp]\!],[\![\dQ_h\xi]\!]}_e\Bigg|\\
        &=\Bigg|\sumE h\seq{[\![\dQ_hp-p]\!],[\![\dQ_h\xi-\xi]\!]}_e\Bigg|\\
        &\le C\left(\sumT h\|p-\dQ_hp\|_{\partial T}^2\right)^{\frac{1}{2}}\left(\sumT h\|\xi-\dQ_h\xi\|_{\partial T}^2\right)^{\frac{1}{2}}\\
        &\le Ch^{k+1}\|p\|_k\|\xi\|_1.
    \end{align*}
    According to (\ref{CDG:Ineq_l4})-(\ref{CDG:Ineq_s}) and the above five estimates, it follows that
    \begin{align*}
        \|\be_h\|^2\le (Ch(\trb{\be_h}+\|\eps_h\|_h)+Ch^{k+1}(\|\bu\|_{k+1}+\| p\|_k))(\|\bpsi\|_2+\|\xi\|_1).
    \end{align*}
    With the regularity assumption (\ref{CDG:pri_cond}) and $H^1$ error estimates (\ref{CDG:Ineq_H11}), we have
    \begin{align*}
        \|\be_h\|^2\le Ch^{k+1}(\|\bu\|_{k+1}+\| p\|_k)\|\be_h\|,
    \end{align*}
    which implies (\ref{CDG:Ineq_L2}).
\end{proof}
\section{Numerical Experiments} \label{EXAMPLE}
In this section, we present several examples in two dimensional domains to verify the stability and order of convergence established in Section \ref{CDG:ERROR_ESTIMATE}.

As before, let $(\bu;p)$ be the exact solution of (\ref{Eq:Brinkman1})-(\ref{Eq:Brinkman3}), $(\bu_h;p_h)$ be the solution of (\ref{CDG:algo1})-(\ref{CDG:algo2}). Denote  $\be_h=Q_h\bu-\bu_h$ and $\eps_h=\dQ_hp-p_h$.

\subsection{Example 1}
Taking $\Omega=(0,1)\times(0,1)$, the exact solution is given as follows:
\begin{align*}
    \bu=\left(\begin{matrix}sin(2\pi x)cos(2\pi y)\\-cos(2\pi x)sin(2\pi y)\end{matrix}\right), \quad p=x^2y^2-\frac{1}{9}.  
\end{align*}
Consider the following permeability
\begin{align*}
    \kappa^{-1}=a(sin({2\pi x})+1.1), 
\end{align*}
where $a$ is a given positive constant. According to the above parameters, the momentum source term $\bf$ and the boundary value $\bg=\bu|_{\partial\Omega}$ can be calculated.

For uniform triangular partition, we choose $k=1$, $\mu=1$, $0.01$ and $a=1$, $10^4$. Table \ref{CDGEX1:1_1_1}-\ref{CDGEX1:1_0.01_10^4} show the errors and orders of convergence accordingly. For uniform rectangular partition and polygonal partition, we choose $k=2$, $3$, $\mu=10^4$ and $a=1$. Table \ref{rectCDGEX1:2_1_10^4}-\ref{polyCDGEX1:3_1_10^4} show the errors and orders of convergence accordingly.

As can be seen from the data in the tables, the numerical experiment results are consistent with the theoretical analysis, and both reach the optimal order of convergence. At the same time, the accuracy and stability of the numerical scheme is verified when the permeability $\kappa$ is highly varying.  

\begin{table}[!htbp]
    \centering
    \caption{Errors and orders of convergence on triangular partition as $k=1,~j=2,~\mu=1,~a=1$}
    \label{CDGEX1:1_1_1}
    \begin{tabular}{|c|c|c|c|c|c|c|c|c|c|c|c|}
        \hline
        $h$ & $\trb{\be_h}$ & order & $\|\be_h\|$ & order & $\|\eps_h\|$ & order \\ \hline
        $1/4$ & 1.5213e+00 &  & 1.02630e-01 &  & 4.3130e-01 &  \\ \hline
        $1/8$ & 6.5903e-01 & 1.2069 & 3.7485e-02 & 1.4531 & 3.1397e-01 & 0.4581 \\ \hline
        $1/16$ & 2.5867e-01 & 1.3493 & 1.0962e-02 & 1.7738 & 1.3347e-01 & 1.2341 \\ \hline
        $1/32$ & 1.0300e-01 & 1.3285 & 2.8870e-03 & 1.9249 & 4.7939e-02 & 1.4773 \\ \hline
        $1/64$ & 4.3307e-02 & 1.2499 & 7.3057e-04 & 1.9825 & 1.8019e-02 & 1.4116 \\ \hline
        $1/128$ & 1.9323e-02 & 1.1643 & 1.8291e-04 & 1.9979 & 7.5708e-03 & 1.2510 \\ \hline
    \end{tabular}
\end{table}

\begin{table}[!htbp]
    \centering
    \caption{Errors and orders of convergence on triangular partition as $k=1,~j=2,~\mu=0.01,~a=1$}
    \label{CDGEX1:1_0.01_1}
    \begin{tabular}{|c|c|c|c|c|c|c|c|c|c|c|c|}
        \hline
        $h$ & $\trb{\be_h}$ & order & $\|\be_h\|$ & order & $\|\eps_h\|$ & order \\ \hline
        $1/4$ & 1.2093e+00 &  & 2.2360e-01 &  & 4.1153e-02 &  \\ \hline
        $1/8$ & 4.8033e-01 & 1.3321 & 7.8492-02 & 1.5103 & 1.8900e-02 & 1.1226 \\ \hline
        $1/16$ & 1.8275e-01 & 1.3941 & 2.2513e-02 & 1.8018 & 8.5580e-03 & 1.1430 \\ \hline
        $1/32$ & 7.1280e-02 & 1.3583 & 5.8547e-03 & 1.9431 & 3.9645e-03 & 1.1101 \\ \hline
        $1/64$ & 2.9236e-02 & 1.2858 & 1.4752e-03 & 1.9886 & 1.8795e-03 & 1.0768 \\ \hline
        $1/128$ & 1.2709e-02 & 1.2019 & 3.6902e-04 & 1.9992 & 9.1017e-04 & 1.0462 \\ \hline
    \end{tabular}
\end{table}

\begin{table}[!htbp]
    \centering
    \caption{Errors and orders of convergence on triangular partition as $k=1,~j=2,~\mu=1,~a=10^4$}
    \label{CDGEX1:1_1_10^4}
    \begin{tabular}{|c|c|c|c|c|c|c|c|c|c|c|c|}
        \hline
        $h$ & $\trb{\be_h}$ & order & $\|\be_h\|$ & order & $\|\eps_h\|$ & order \\ \hline
        $1/4$ & 2.2860e+00 &  & 2.9740e-02 &  & 1.6643e-01 &  \\ \hline
        $1/8$ & 3.6162e-01 & 2.6603 & 3.9858e-03 & 2.8995 & 1.9585e-01 & -0.2349 \\ \hline
        $1/16$ & 1.5733e-01 & 1.2007 & 1.4717e-03 & 1.4374 & 1.8216e-01 & 0.1046 \\ \hline
        $1/32$ & 8.4025e-02 & 0.9049 & 5.4615e-04 & 1.4301 & 1.3765e-01 & 0.4042 \\ \hline
        $1/64$ & 4.1075e-02 & 1.0326 & 1.7638e-04 & 1.6306 & 7.6546e-02 & 0.8466 \\ \hline
        $1/128$ & 1.9215e-02 & 1.0960 & 5.0498e-05 & 1.8044 & 3.0815e-02 & 1.3127 \\ \hline
    \end{tabular}
\end{table}

\begin{table}[!htbp]
    \centering
    \caption{Errors and orders of convergence on triangular partition as $k=1,~j=2,~\mu=0.01,~a=10^4$}
    \label{CDGEX1:1_0.01_10^4}
    \begin{tabular}{|c|c|c|c|c|c|c|c|c|c|c|c|}
        \hline
        $h$ & $\trb{\be_h}$ & order & $\|\be_h\|$ & order & $\|\eps_h\|$ & order \\ \hline
        $1/4$ & 1.0369e+00 &  & 6.7825e-02 &  & 1.2495e-01 &  \\ \hline
        $1/8$ & 4.4656e-01 & 1.2154 & 2.0320e-02 & 1.7389 & 8.4347e-02 & 0.5670 \\ \hline
        $1/16$ & 1.7634e-01 & 1.3405 & 7.3711e-03 & 1.4630 & 4.7227e-02 & 0.8367 \\ \hline
        $1/32$ & 6.9966e-02 & 1.3336 & 2.4262e-03 & 1.6032 & 1.9975e-02 & 1.2414 \\ \hline
        $1/64$ & 2.9008e-02 & 1.2702 & 6.8638e-04 & 1.8216 & 6.4779e-03 & 1.6246 \\ \hline
        $1/128$ & 1.2675e-02 & 1.1945 & 1.7914e-04 & 1.9379 & 1.9103e-03 & 1.7618 \\ \hline
    \end{tabular}
\end{table}

\begin{table}[!htbp]
    \centering
    \caption{Errors and orders of convergence on rectangular partition as $k=2,~j=5,~\mu=1,~a=10^4$}
    \label{rectCDGEX1:2_1_10^4}
    \begin{tabular}{|c|c|c|c|c|c|c|c|c|c|c|c|}
         \hline
         $h$ & $\trb{\be_h}$ & order & $\|\be_h\|$ & order & $\|\eps_h\|$ & order \\ \hline
         $1/4$ & 1.4777e+00 &  & 1.7715e-02 &  & 5.4279e-01 &  \\ \hline
         $1/8$ & 2.5250e-01 & 2.5490 & 1.7706e-03 & 3.3227 & 1.2107e-01 & 2.1645 \\ \hline
         $1/16$ & 5.5071e-02 & 2.1969 & 2.4002e-04 & 2.8830 & 3.2960e-02 & 1.8771 \\ \hline
         $1/32$ & 1.2294e-02 & 2.1633 & 3.2057e-05 & 2.9045 & 6.2549e-03 & 2.3976 \\ \hline
         $1/64$ & 2.7958e-03 & 2.1367 & 3.8025e-06 & 3.0756 & 1.0068e-03 & 2.6352 \\ \hline
    \end{tabular}
\end{table}

\begin{table}[!htbp]
    \centering
    \caption{Errors and orders of convergence on rectangular partition as $k=3,~j=6,~\mu=1,~a=10^4$}
    \label{rectCDGEX1:3_1_10^4}
    \begin{tabular}{|c|c|c|c|c|c|c|c|c|c|c|c|}
         \hline
         $h$ & $\trb{\be_h}$ & order & $\|\be_h\|$ & order & $\|\eps_h\|$ & order \\ \hline
         $1/4$ & 6.4024e-01 &  & 5.3190e-03 &  & 4.8987e-01 &  \\ \hline
         $1/8$ & 7.5994e-02 & 3.0746 & 3.6500e-04 & 3.8652 & 4.0092e-02 & 3.6110 \\ \hline
         $1/16$ & 7.9936e-03 & 3.2490 & 2.2508e-05 & 4.0194 & 2.7254e-03 & 3.8788 \\ \hline
         $1/32$ & 8.2601e-04 & 3.2746 & 1.3259e-06 & 4.0854 & 1.8280e-04 & 3.8981 \\ \hline
         $1/64$ & 8.8561e-05 & 3.2214 & 7.8119e-08 & 4.0852 & 1.1620e-05 & 3.9756 \\ \hline
    \end{tabular}
\end{table}

\begin{table}[!htbp]
    \centering
    \caption{Errors and orders of convergence on polygonal partition as $k=2,~j=8,~\mu=1,~a=10^4$}
    \label{polyCDGEX1:2_1_10^4}
    \begin{tabular}{|c|c|c|c|c|c|c|c|c|c|c|c|}
         \hline
         $h$ & $\trb{\be_h}$ & order & $\|\be_h\|$ & order & $\|\eps_h\|$ & order \\ \hline
         $1/4$ & 1.3518e+00 &  & 1.5787e-02 &  & 5.9253e-01 &  \\ \hline
         $1/8$ & 3.4537e-01 & 1.9687 & 2.3238e-03 & 2.7642 & 1.8185e-01 & 1.7041 \\ \hline
         $1/16$ & 1.0534e-01 & 1.7130 & 3.8524e-04 & 2.5927 & 3.8610e-02 & 2.2357 \\ \hline
         $1/32$ & 2.7395e-02 & 1.9431 & 5.2055e-05 & 2.8876 & 4.9117e-03 & 2.9747 \\ \hline
         $1/64$ & 7.0084e-03 & 1.9668 & 6.6314e-06 & 2.9727 & 7.1932e-04 & 2.7715 \\ \hline
    \end{tabular}
\end{table}

\begin{table}[!htbp]
    \centering
    \caption{Errors and orders of convergence on polygonal partition as $k=3,~j=9,~\mu=1,~a=10^4$}
    \label{polyCDGEX1:3_1_10^4}
    \begin{tabular}{|c|c|c|c|c|c|c|c|c|c|c|c|}
         \hline
         $h$ & $\trb{\be_h}$ & order & $\|\be_h\|$ & order & $\|\eps_h\|$ & order \\ \hline
         $1/4$ & 6.2524e-01 &  & 5.2453e-03 &  & 4.0361e-01 &  \\ \hline
         $1/8$ & 7.9629e-02 & 2.9730 & 3.7527e-04 & 3.8050 & 3.5666e-02 & 3.5003 \\ \hline
         $1/16$ & 8.0684e-03 & 3.3029 & 2.4399e-05 & 3.9431 & 2.1461e-03 & 4.0548 \\ \hline
         $1/32$ & 8.2383e-04 & 3.2919 & 1.5631e-06 & 3.9644 & 1.5572e-04 & 3.7847 \\ \hline
         $1/64$ & 8.9162e-05 & 3.2078 & 9.5552e-08 & 4.0319 & 1.1355e-05 & 3.7775 \\ \hline
    \end{tabular}
\end{table}

The rest of examples in this section have the following setting:
\begin{align*}
 k=1,\quad\Omega=(0,1)\times(0,1),\quad\mu=0.01,\quad\bf=\left(\begin{matrix}0\\0\end{matrix}\right),\quad\bg=\left(\begin{matrix}1\\0\end{matrix}\right).
\end{align*}

\subsection{Example 2} \label{EX2}
In this example, the permeability coefficient $\kappa$ is selected as the piecewise constant function with highly varying. The profile of the permeability inverse is shown in Fig. \ref{EX2_kappa}. As we know, this example has no analytic solutions.

In Fig. \ref{EX2_up}, a 128$\times$128 rectangular partition is used to solve this example. The profiles of the pressure and the two components of the velocity for CDG are plotted in Fig. \ref{CDGEX2_p}-\ref{CDGEX2_u2}

\begin{figure}[!htbp]
    \centering
    \subfigure[Profile of $\kappa^{-1}$]{
    \includegraphics[width=0.45\textwidth]{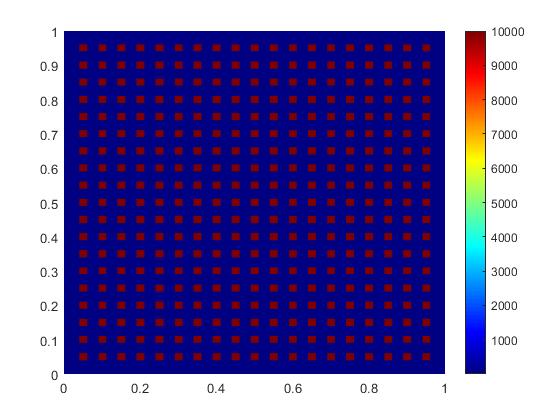} \label{EX2_kappa}
    }
    \subfigure[Profile of p]{
    \includegraphics[width=0.45\textwidth]{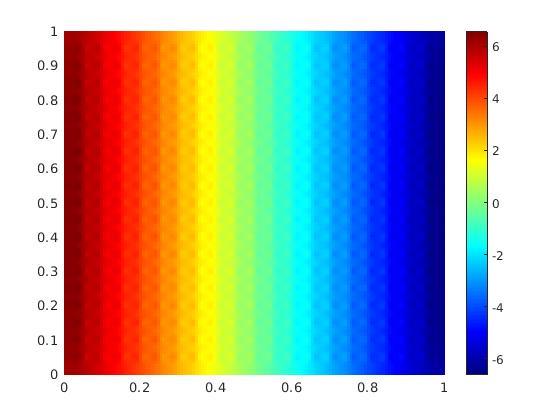} \label{CDGEX2_p}
    }
    \subfigure[Profile of first component of $\bu$]{
    \includegraphics[width=0.45\textwidth]{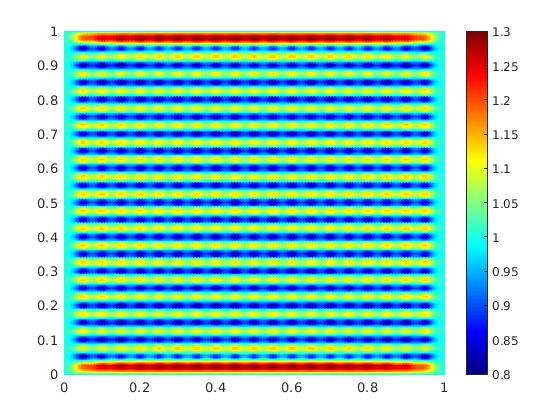} \label{CDGEX2_u1}
    }
    \subfigure[Profile of second component of $\bu$]{
    \includegraphics[width=0.45\textwidth]{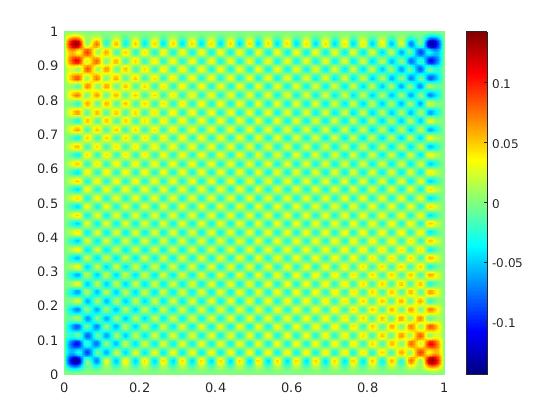} \label{CDGEX2_u2}
    }
    \caption{Profiles of $\kappa^{-1}$ and numerical solution in Ex. 2}
    \label{EX2_up}
\end{figure}

\subsection{Example 3}
In this example, we choose a vuggy medium with the permeability coefficient $\kappa$ highly varying. The profile of the permeability inverse is plotted in Fig. \ref{EX3_kappa}. For solving this example, a 128$\times$128 rectangular partition is used. And the pressure obtained by CDG method is present in Fig. \ref{CDGEX3_p}. The velocity profiles are shown in Fig. \ref{CDGEX3_u1}-\ref{CDGEX3_u2}.

\begin{figure}[!htbp]
    \centering
    \subfigure[Profile of $\kappa^{-1}$]{
    \includegraphics[width=0.45\textwidth]{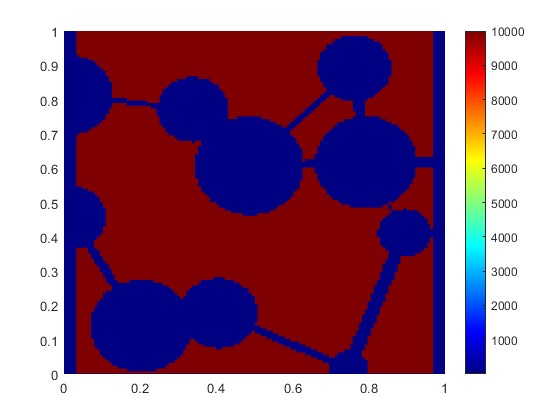} \label{EX3_kappa}
    }
    \subfigure[Profile of p]{
    \includegraphics[width=0.45\textwidth]{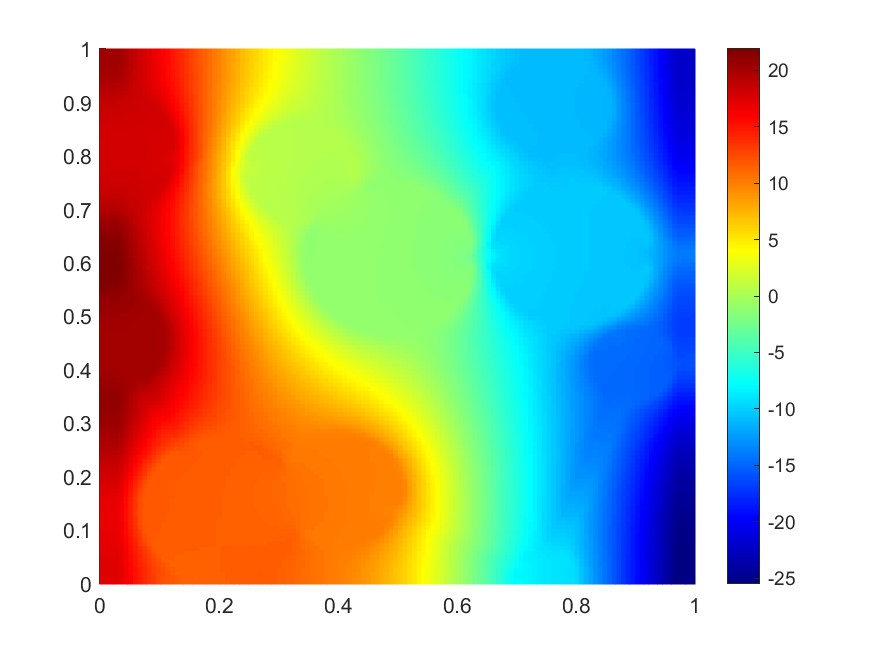} \label{CDGEX3_p}
    }
    \subfigure[Profile of first component of $\bu$]{
    \includegraphics[width=0.45\textwidth]{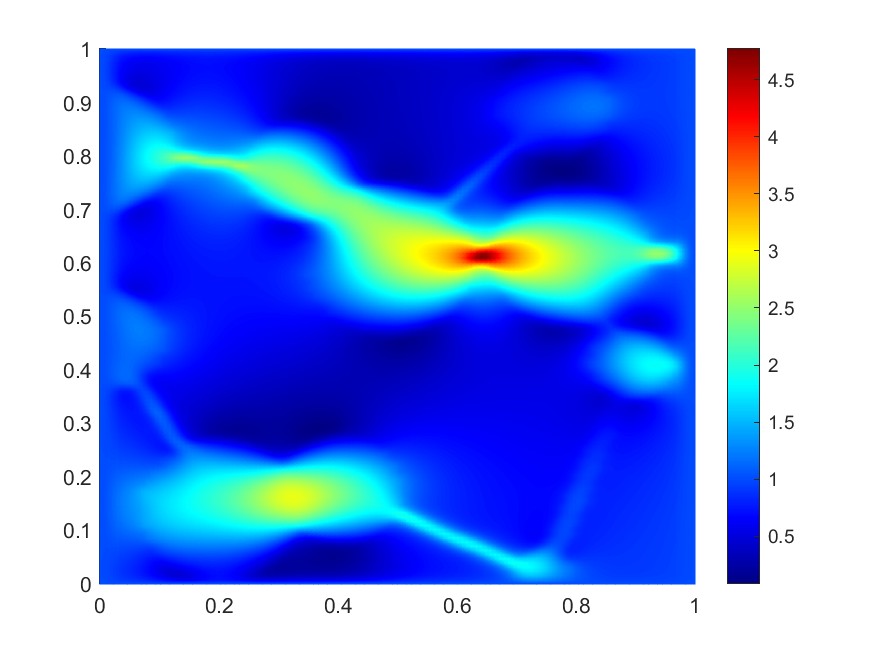} \label{CDGEX3_u1}
    }
    \subfigure[Profile of second component of $\bu$]{
    \includegraphics[width=0.45\textwidth]{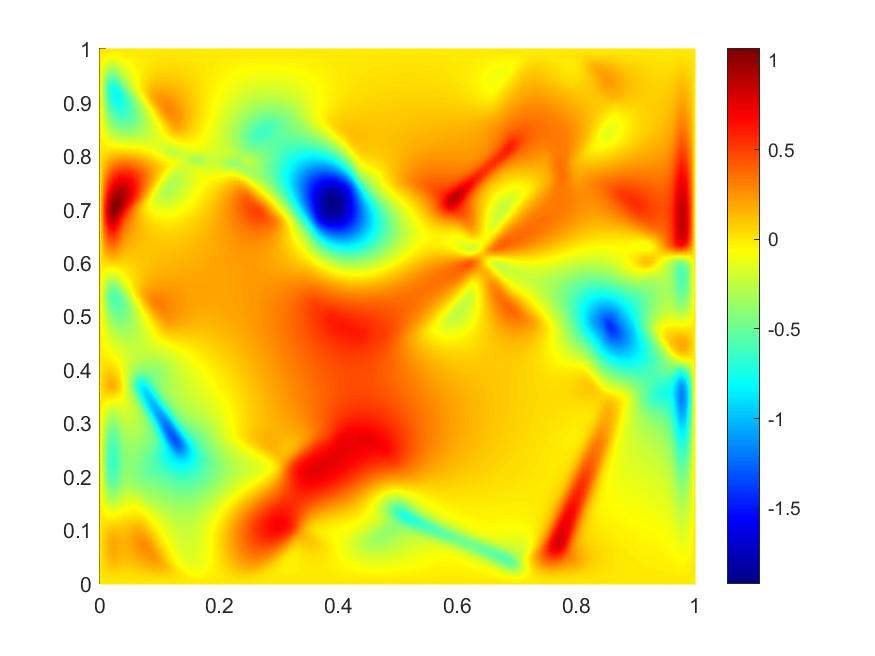} \label{CDGEX3_u2}
    }
    \caption{Profiles of $\kappa^{-1}$ and numerical solution in EX. 3}
    \label{EX3_up}
\end{figure}

\subsection{Example 4}
The fluid flowing in a fibrous material can also be described by Brinkman equations. Fig. \ref{EX4_kappa} shows the inverse of permeability in a common fibrous material. The parameters are the same as in the previous example. We can get the corresponding pressure in Fig. \ref{CDGEX4_p} and velocity in Fig. \ref{CDGEX4_u1}-\ref{CDGEX4_u2} by CDG method.

\begin{figure}[!htbp]
    \centering
    \subfigure[Profile of $\kappa^{-1}$]{
    \includegraphics[width=0.45\textwidth]{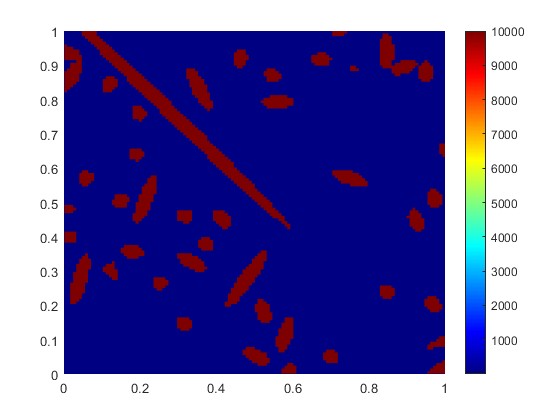} \label{EX4_kappa}
    }
    \subfigure[Profile of p]{
    \includegraphics[width=0.45\textwidth]{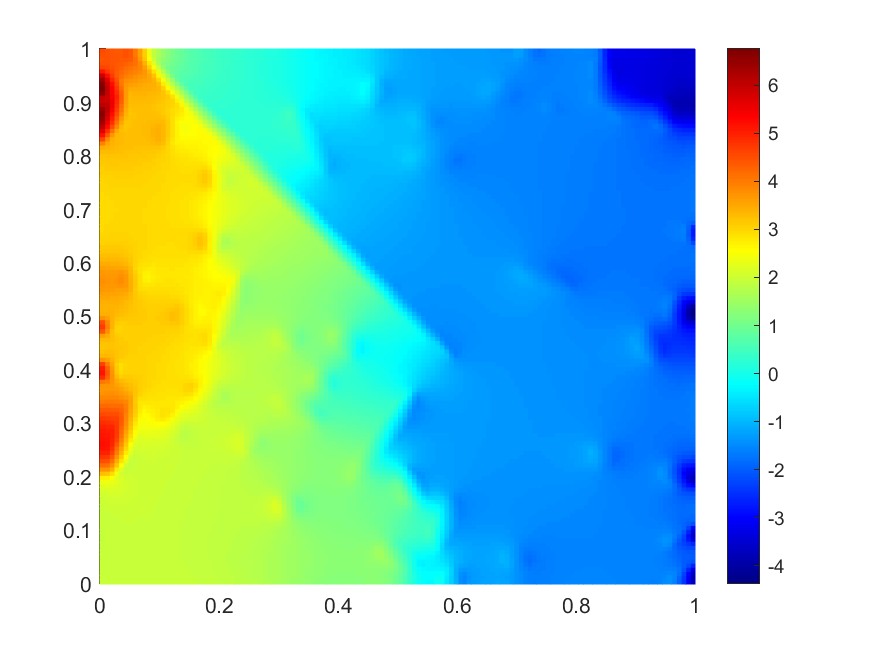} \label{CDGEX4_p}
    }
    \subfigure[Profile of first component of $\bu$]{
    \includegraphics[width=0.45\textwidth]{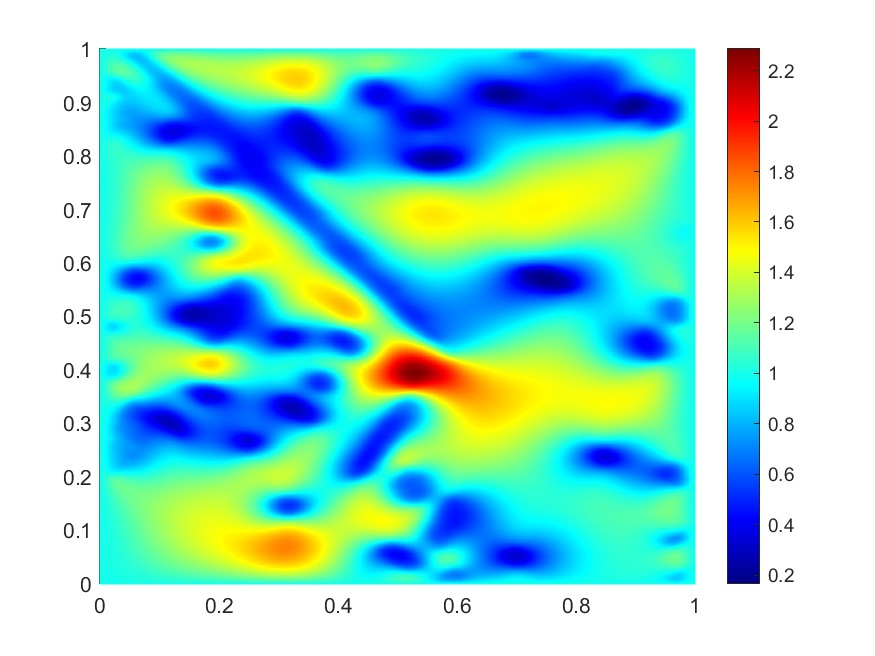} \label{CDGEX4_u1}
    }
    \subfigure[Profile of second component of $\bu$]{
    \includegraphics[width=0.45\textwidth]{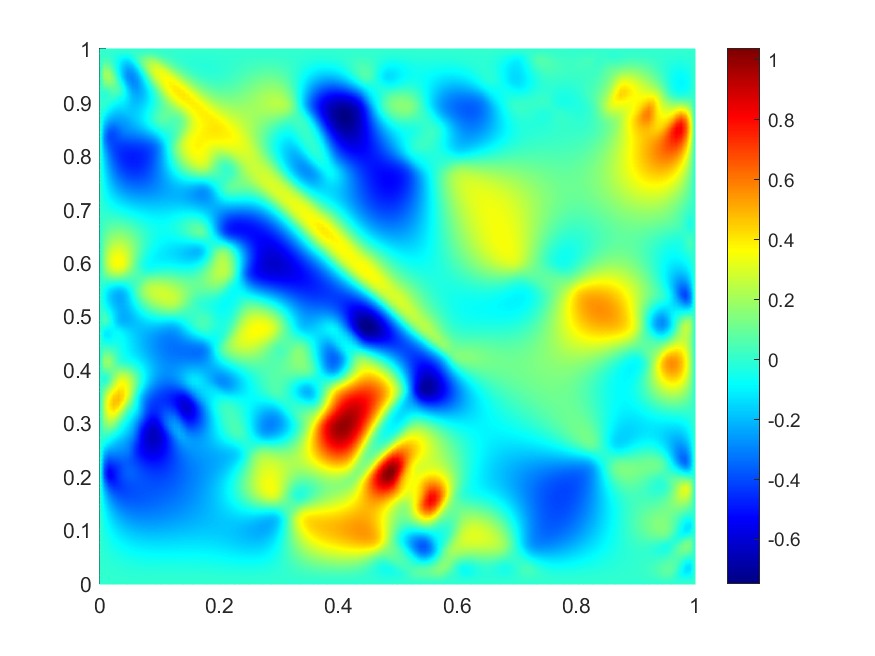} \label{CDGEX4_u2}
    }
    \caption{Profiles of $\kappa^{-1}$ and numerical solution in EX. 4}
    \label{EX4_up}
\end{figure}
\begin{appendices} 
\section{Some Inequality Estimates} \label{APPENDIX}
In this Appendix, we provide inequalities for projection operators $Q_h$, $\bQ_h$, $\dQ_h$ and inequality estimates used in the previous paper.
\begin{lemma}
    \cite{MESH} Let $\T_h$ be a shape regular partition of $\Omega$, $\bw\in[H^{k+1}(\Omega)]^d$ and $\rho\in H^k(\Omega)$. Then we have the following projection inequalities
    \begin{align}
        \sumT\|\bw-Q_h\bw\|^2_T&\le Ch^{2(k+1)}\|\bw\|^2_{k+1}, \label{proj_ineq1}\\
        \sumT\|\nabla\bw-\bQ_h(\nabla\bw)\|^2_T&\le Ch^{2k}\|\bw\|^2_{k+1}, \label{proj_ineq2}\\
        \sumT\|\rho-\dQ_h\rho\|^2_T&\le Ch^{2k}\|\rho\|_k^2. \label{proj_ineq3}
    \end{align}
    where $C$ is a constant independent of the size of mesh $h$ and the functions $\bw$ and $\rho$. 
\end{lemma}

Let $T$ be a cell with $e$ as an edge/face. For any function $\rho\in H^1(T)$, the following trace inequality has been proved to be valid in \cite{MESH}:
\begin{align}
    \|\rho\|^2_e\le C(h^{-1}_T\|\rho\|^2_T+h_T\|\nabla\rho\|^2_T). \label{trace_ineq}
\end{align}

Furthermore, if $\rho\in P_k(T)$, we have the following inverse inequality
\begin{align}
    \|\nabla\rho\|_T\le Ch_T^{-1}\|\rho\|_T. \label{inverse_ineq}
\end{align}

\begin{lemma}
    For any $\bv\in V_h$, the following inequality holds true
    \begin{align}
        \sum_{e\in\E_h}h^{-1}\|[\bv]\|^2_e\le C\trb{\bv}^2, \label{CDG:Ineq_w}
    \end{align}
    where $C$ is a positive constant.
\end{lemma}

The proof of this lemma is given in Lemma 3.2 in \cite{JMIN1}.

\begin{lemma} \label{CDG:lemma_error_ineq}
    For any $\bw\in[H^{k+1}(\Omega)]^d$, $r\in H^k(\Omega)$ $\bv\in V_h$ and $q\in W_h$. Then we have 
    \begin{align}
        |l_1(\bw,\bv)|&\le Ch^k\|\bw\|_{k+1}\trb{\bv}, \label{CDG:Ineq_l4}\\
        |l_2(\bw,\bv)|&\le Ch^k\|\bw\|_{k+1}\trb{\bv}, \label{CDG:Ineq_l5}\\
        |l_3(r,\bv)|&\le Ch^k\| r\|_k\trb{\bv}, \label{CDG:Ineq_l6}\\
        |l_4(\bw,q)|&\le Ch^k\|\bw\|_{k+1}\|q\|_h, \label{CDG:Ineq_l7}\\
        |s(\dQ_hr,q)|&\le Ch^k\|r\|_k\|q\|_h.\label{CDG:Ineq_s} 
    \end{align}
\end{lemma}
\begin{proof}
    Using the definition of $\nabla_w$, the Cauchy-Schwarz inequality, the trace inequality (\ref{trace_ineq}) and the projection inequality (\ref{proj_ineq1}), we obtain
    \begin{align*}
        |l_1(\bw,\bv)|&=\Bigg|\sumT(\nabla_w(\bw-Q_h\bw),\nabla_w\bv)_T\Bigg|\\
        &=\Bigg|\sumT(-(\bw-Q_w\bw,\nabla\cdot\nabla_w\bv)_T+\seq{\set{\bw-Q_w\bw},\nabla_w\bv\cdot\bn}_{\partial T})\Bigg|\\
        &=\Bigg|\sumT((\nabla(\bw-Q_w\bw),\nabla_w\bv)_T+\seq{ Q_w\bw-\set{Q_w\bw},\nabla_w\bv\cdot\bn}_{\partial T})\Bigg|\\
        &=\Bigg|\sumT((\nabla(\bw-Q_w\bw),\nabla_w\bv)_T+\seq{ Q_w\bw-\bw+\set{\bw-Q_w\bw},\nabla_w\bv\cdot\bn}_{\partial T})\Bigg|\\
        &\le\sumT(\|\nabla(\bw-Q_w\bw)\|_T\|\nabla_w\bv\|_T+Ch^{-\frac{1}{2}}\|\bw-Q_w\bw\|_{\partial T}\|\nabla_w\bv\|_T)\\
        &\le\left(\sumT(\|\nabla(\bw-Q_w\bw)\|_T+Ch^{-1}\|\bw-Q_w\bw\|_T)^2\right)^{\frac{1}{2}}\left(\sumT\|\nabla_w\bv\|_T^2\right)^{\frac{1}{2}}\\
        &\le Ch^k\|\bw\|_{k+1}\trb{\bv}.
    \end{align*}
    Combining (\ref{CDG:Ineq_w}), (\ref{v_averjump}) and the projection inequality (\ref{proj_ineq2}) gives
    \begin{align*}
        |l_2(\bw,\bv)|&=\Bigg|\sumT\seq{(\nabla\bw-\bQ_h\nabla\bw)\cdot\bn,\bv-\set{\bv}}_{\partial T}\Bigg|\\
        &\le C\left(\sumT h\|\nabla\bw-\bQ_h\nabla\bw\|_{\partial T}^2\right)^{\frac{1}{2}}\left(\sum_{e\in\E_h}h^{-1}\|[\bv]\|_e^2\right)^{\frac{1}{2}}\\
        &\le Ch^k\|\bw\|_{k+1}\trb{\bv}.
    \end{align*}
    According to the projection inequality (\ref{proj_ineq3}), we get
    \begin{align*}
        |l_3(r,\bv)|&=\Bigg|\sumT\seq{r-\dQ_hr,\bv\cdot\bn}_{\partial T}\Bigg|\\
        &\le C\left(\sumT h\|r-\dQ_hr\|_{\partial T}^2\right)^{\frac{1}{2}}\left(\sum_{e\in\E_h}h^{-1}\|[\bv]\|_e^2\right)^{\frac{1}{2}}\\
        &\le Ch^k\|r\|_k\trb{\bv}.
    \end{align*}
    Similarly, with (\ref{q_averjump}), then
    \begin{align*}
        |l_4(\bw,q)|&=\Bigg|\sumT\seq{(\bu-Q_h\bu)\cdot\bn,q-\set{q}}_{\partial T}\Bigg|\\
        &\le C\left(\sumT h^{-1}\|\bw-Q_h\bw\|_{\partial T}^2\right)^{\frac{1}{2}}\left(\sumE h\|[\![q]\!]\|_e^2\right)^{\frac{1}{2}}\\
        &\le Ch^k\|\bw\|_{k+1}\|q\|_h.
    \end{align*}
    Using the definition of $s(\cdot,\cdot)$, we get
    \begin{align*}
        |s(\dQ_hr,q)|&=\Bigg|\sumE h\seq{[\![\dQ_hr]\!],[\![q]\!]}_e\Bigg|\\
        &=\Bigg|\sumE h\seq{[\![\dQ_hr-r]\!],[\![q]\!]}_e\Bigg|\\
        &\le C\left(\sumT h\|r-\dQ_hr\|_{\partial T}^2\right)^{\frac{1}{2}}\left(\sumE h\|[\![q]\!]\|_e^2\right)^{\frac{1}{2}}\\
        &\le Ch^k\|r\|_k\|q\|_h, 
    \end{align*}
    which completes the proof.
\end{proof}

\end{appendices}

\section*{Acknowledgments}
This work was supported by National Natural Science Foundation of China (Grant No. 1901015, 12271208), and Key Laboratory of Symbolic Computation and Knowledge Engineering of Ministry of Education, Jilin University. 
We sincerely thank the anonymous reviewers for their insightful comments, which have helped improve the quality of this paper.






\end{document}